\documentclass[11pt,letterpaper]{article}
\usepackage{fullpage}
\usepackage{amsmath}
\usepackage{amsfonts}
\usepackage{amssymb}
\usepackage{amsthm}
\usepackage{tikz-cd}
\usepackage{cleveref}
\usepackage{xcolor}
\usepackage{comment}
\definecolor{orange}{rgb}{1,0.5,0}

\title{\bf Gromov-Witten theory under degree-4 Type II Extremal Transitions}
\author{Rongxiao Mi}
\date{\vspace{-5ex}}

\newtheorem{thm}{Theorem}[section]
\newtheorem{lemma}[thm]{Lemma}

\newtheorem{conj}[thm]{Conjecture}

\theoremstyle{definition}

\newtheorem{definition}[thm]{Definition}

\def\th@remark{%
	\thm@headfont{\bfseries}%
	\normalfont 
	\thm@preskip\topsep \divide\thm@preskip\tw@
	\thm@postskip\thm@preskip
}
\theoremstyle{remark}
\newtheorem{remark}[thm]{Remark}
\numberwithin{equation}{section}

\newcommand{\ev}{\mathrm{ev}}

\newcommand{\cM}{\overline{\mathcal{M}}}
\newcommand{\cL}{\mathcal{L}}
\newcommand{\cH}{\mathcal{H}}
\newcommand{\cV}{\mathcal{V}}
\newcommand{\ZZ}{\mathbb{Z}}
\newcommand{\CC}{\mathbb{C}}
\newcommand{\cO}{\mathcal{O}}
\newcommand{\PP}{\mathbb{P}}
\newcommand{\DD}{\mathbb{D}}
\newcommand{\NN}{\mathbb{N}}

\usetikzlibrary{decorations.pathmorphing}

\begin{document}
\maketitle

\begin{abstract}
In this article, we study the change of genus zero Gromov-Witten invariants under Type II extremal transitions in degree 4.
\end{abstract}

\tableofcontents

\section{Introduction}
There has been a long-standing interest in understanding how Gromov-Witten theories change under various types of surgeries. For one reason, surgeries play a significant role in linking two geometrically or topologically distint Calabi-Yau 3-folds. For another, it gives some insight on how one can extend the picture of mirror symmetry to a larger class of Calabi-Yau 3-folds. Over two decades ago, Li-Ruan \cite{MR1839289} initiated a program to study the change of Gomov-Witten theory under flops and conifold transitions, which has led to many important new discoveries and exciting developments of degeneration techinques in Gromov-Witten theory. More specifically, they have shown that a flop between two Calabi-Yau 3-folds induces an isomorphism between quantum cohomology rings, up to analytic continuation over the extended K\"ahler moduli space. Later, Lee-Lin-Wang established the quantum invariance for higher dimensional flops in various settings \cite{MR2680420, MR3568338, MR3568339}. Their work highlights the fact that two Gromov-Witten theories related by a flop are essentially equivalent via an appropriate analytic continuation of the quantum variables.

In contrast with flops, the change of Gromov-Witten theory under extremal transitions are usually subtler and more complicated. There is apparently an asymmetric nature alluded in forming the picture of extremal transitions: one side is obtained by resolution of singularity, yet the other side is by a smoothing. For this reason, we cannot expect their Gromov-Witten theories to be equivalent. There is not yet an established concensus as to how one can relate their Gromov-Witten theories. What has been a commonplace approach is to compare certain significant structures arising from the genus zero Gromov-Witten theory, as were adopted in many works \cite{MR3568645,MR3751817}. In \cite{Mi:2017zwz}, the author takes a new approach to relate the (ambient part) quantum $D$-modules (denoted by $\cH(X)$ for a given variety $X$) in the case of cubic extremal transitions, and proposes the following conjecture:

\begin{conj}\label{myconj}
	Suppose two smooth projective varieties $X$ and $Y$ are related by a primitive extremal transition. then one may perform analytic continuation of $\cH(X)$ over the extended K\"ahler moduli to obtain a $D$-module $\bar{\cH}(X)$, then there is a divisor $E$ and a submodule $\bar{\cH}^E(X)\subseteq \bar{\cH}(X)$ with maximum trivial $E$-monodromy such that
	$$\bar{\cH}^E(X)|_E\simeq \cH(Y),$$
	where $\bar{\cH}^E(X)|_E$ is the restriction to $E$.
\end{conj}

The above conjecture rests on several observations: first, the quantum $D$-modules of $X$ and $Y$ usually involves different quantum variables. To compare their quantum $D$-modules, one have to find a way to relate their quantum variables. This is usually done by analytic continuation, and in this way we obtain a $D$-modules $\bar{\cH}(X)$. Moreover, there is a discrepency between the rank of $\bar{\cH}(X)$ and $\cH(Y)$, which leads us to think about identifying $\cH(Y)$ as a submodule of $\bar{\cH}(X)$ after certain restriction along a transition divisor $E$. There is also the issue of monodromy involved when one wants to make sense of the restriction of a quantum $D$-module. For this reason, we need to require the submodule admit only trivial monodromy around the transition divisor in question. Finally, it turns out that in many cases, this submodule should have maximum trivial $E$-monodromy.

This paper continues our study on the change of Gromov-Witten theory under Type II transitions, i.e. the birational contraction involved in the transition is given by contracting a divisor to a point. If one reuiqres that the variety obtained after performing birational contraction has only isolated complete intersection singularity, then the exceptional divisor is a Del-Pezzo surface of degree either 3 or 4 (cf. \cite{MR1464900}). For the degree-3 case, \Cref{myconj} has already been verified in \cite{Mi:2017zwz} for both the local model and a global example involving two Calabi-Yau 3-folds related by such a transition. 

 In this paper, we will deal with the remaining case above where the divisor in question is a Del-Pezzo surface of degree 4. In particular, we will verify \Cref{myconj} for the local model as well as two global examples. In the future work, we expect more general cases would be proved using degeneration techinques.\\
~\\
\textbf{Type II transition in degree 4.} Let us first introduce the setup of the Type II extremal transition in degree 4: We have a pair of Calabi-Yau 3-folds $(X,Y)$ related in the following diagram
\begin{center}
	\begin{tikzcd}
		X\arrow{d}{\pi}\arrow[dr,bend left, dotted]&\\
		\overline{Y}\ar[squiggly]{r} &Y
	\end{tikzcd}
\end{center}
in which $\pi:X\to \overline{Y}$ is a birational contraction of a divisor $E\hookrightarrow X$ to a point $p\in \overline{Y}$, and going from $\overline{Y}$ to $Y$ is given by smoothing out the singularity at $p$. We further require that $E$ be a Del-Pezzo surface of degree 4 in $\PP^4$, and the curve classes on $E$ generate an extremal ray in the Mori cone of $X$. In this case, the singularity aournd $p$ is a complete intersection of two equations with quadratic leading terms. We call the process of going form $X$ to $Y$ \emph{a Type II extremal transition in degree 4} (degree-4 transition for short). We are going to consider the following local model and global examples.\\
~\\
\textbf{The Local Model.} Let us focus on the local picture of the surgery. Let $E$ be the degree-4 Del-Pezzo surface inside a Calabi-Yau 3-fold $X$, which gets contracted to a point under $\pi$. Consider a tubular neighborhood around $E$, which is identified with its normal bundle $N_{E/X}$. Since the ambient variety is a Calabi-Yau 3-fold, we also have $N_{E/X}\simeq K_E$. Under the birational morphism $\pi: X\to \overline{Y}$, the divisor is contracted to a point with singularity given by a complete intersection of two equations with quadratic leading terms. Smoothing out the singularity locally yields a (2,2) complete intersection in $\PP^5$. To study the Gromov-Witten theory of this local picutre, we usually compactify $N_{E/X}\simeq K_E$ and take $X':=\PP(K_E\oplus \cO)$, while $Y':=(2,2)\mbox{ complete intersection in }\PP^5$. The transition from $X'$ to $Y'$ is the \emph{local model} of the Type II extremal transition in degree 4.\\
~\\
\textbf{Global Examples.} We will give two global examples. Here "global" means both sides of the transition are Calabi-Yau 3-folds. Let  $f(x_1,x_2,x_3,x_4,x_5), g(x_1,x_2,x_3,x_4,x_5)$ be two quadratic homogenous polynomials that form a complete intersection. In the first global example, we take
$Y_1:=(2,4)\mbox{ complete intersection in }\PP^5$ given by the following two equations:
\[f(x_1,x_2,x_3,x_4,x_5)+tx_0^2=0,\]
\[x_0^2(g(x_1,x_2,x_3,x_4,x_5)+tx_0^2)+x_0g_1(x_1,x_2,x_3,x_4,x_5)+g_2(x_1,x_2,x_3,x_4,x_5)=0.\]
where $g_i$ is a generic homogenous polynomial in degree $(2+i)$. A simple deformation of $Y_1$ is the following singular variety $\overline{Y_1}$ in $\PP^5$ given by
\[f(x_1,x_2,x_3,x_4,x_5)=0,\]
\[x_0^2g(x_1,x_2,x_3,x_4,x_5)+x_0g_1(x_1,x_2,x_3,x_4,x_5)+g_2(x_1,x_2,x_3,x_4,x_5)=0.\]
By choosing $g_1$ and $g_2$ sufficiently general, one may assume $\overline{Y_1}$ has a unique singularity at $[1:0:0:0:0:0]$ given by a (2,2) complete intersection. Now we take $X_1:=\mbox{Blow-up of }\overline{Y_1}\mbox{ at }[1:0:0:0:0:0]$. It is easy to check that $X_1\to \overline{Y_1}$ is a contraction with exceptional divisor being a (2,2)-complete intersection in $\PP^4$, i.e. a Del-Pezzo surface of degree 4. Thus the passage from $X_1$ to $Y_1$ is a global example.\\
~\\
The second global example is obtained in a similar fashion. We take $Y_2$ to be a $(3,3)$ complete intersection in $\PP^5$ given by the following two equations:
\[x_0(f(x_1,x_2,x_3,x_4,x_5)+tx_0^2)+f_1(x_1,x_2,x_3,x_4,x_5)=0,\]
\[x_0(g(x_1,x_2,x_3,x_4,x_5)+tx_0^2)+g_1(x_1,x_2,x_3,x_4,x_5)=0.\]
where $f_1,g_1$ are both generic cubic homogenous polynomials. A simple deformation of $Y_2$ is the following singular variety $\overline{ Y_2}$ in $\PP^5$ given by
\[x_0f(x_1,x_2,x_3,x_4,x_5)+f_1(x_1,x_2,x_3,x_4,x_5)=0,\]
\[x_0g(x_1,x_2,x_3,x_4,x_5)+g_1(x_1,x_2,x_3,x_4,x_5)=0.\]
By choosing $f_1$ and $g_1$ sufficiently general, one may assume $\overline Y_2$ has a unique singularity at $[1:0:0:0:0:0]$ given by a (2,2) complete intersection. Now we take $X_1:=\mbox{Blow-up of }\overline Y_2\mbox{ at }[1:0:0:0:0:0]$. The passage from $X_2$ to $Y_2$ is another global example that we are interested in.\\
~\\
\textbf{The ambient part quantum $D$-modules.} To every complete intersection $Z$ in a projective toric variety $P$ (or more generally, a GIT quotient), one may associate a $D$-module $\cH(Z)$\footnote{This notation is used when the ambient variety $P$ is understood in the context.} to the pair $(Z,P)$, which encodes the genus zero ambient part Gromov-Witten theory of $Z$.

There are many different ways to define such a quantum $D$-module. For our purpose in this paper, we will use the definition given in \cite{Mi:2017zwz}. Let $Z$ be given as a generic zero section of a vector bundle $\cV$ over $P$. Let $\{T^i\}$ be a basis of $H^*(P)$ with $T^0=1$, and $T^1,\cdots, T^r$ form a basis of $H^2(P)$. We begin with following generating series, introduced in Coates-Givental's work \cite{MR2276766}, which is usually called the twisted $J$-function for the pair $(P,\cV)$.
\[J_\mathrm{big}^\cV(t,z^{-1}):=1+\frac{t}{z}+\sum_{n=0}^\infty\frac{Q^d}{n!}(\ev_{n+1})_*\left(\frac{e(\cV'_{0,n+1,d})}{z(z-\psi_{n+1})}\prod_{i=1}^n\ev^*_i t\right),\]
where $T=\sum t_iT^i$ is a general cohomology class in $P$, $\cV'_{0,n+1,d}$ is the kernel of the evaluation map $R^0\pi_*\ev^*_{n+2}\cV\to \ev_{n+1}^*\cV$ and $\pi:\cM_{0,n+2}(P,\beta)\to \cM_{0,n+1}(P,\beta)$ forgets the last marking.
Setting $t_{r+1}=\cdots=t_m=0$, $Q\equiv 1$ and $q:=(q_i)=(e^{t_i})$ for $i=1,\cdots,r$, we obtain the small $J$-function $J_\mathrm{sm}^\cV(t_0,q,z^{-1})$. 

Let $\DD_q$ be the ring of differential operators generated by $z,\delta_{q_i},q_j^{\pm 1}$ where $\delta_{q_i}$ represents the log differential operators $q_i\partial_{q_i}$.
Then we define the {\em ambient part quantum $D$-module} $\cH(Z)$ to be the cyclic $\DD_q$-module generated by $\mathbf{e}(\cV)J_\mathrm{sm}^\cV(0,q,z^{-1})$. 

We introduce the notion of quantum $D$-module this way because it is more relatable to the Mirror Theorem, which allows us to work directly with the (twisted) $I$-functions. In general, $I$-functions are explicitly cohomology-valued hypergeometric series, which arises as solutions to certain GKZ system attached to the toric data.\\
~\\
\textbf{Main Theorem.} In this paper, we will verify \Cref{myconj} for the local model as well as two global examples mentioned above. Our main result is the following.
\begin{thm}
	\Cref{myconj} holds for the following cases:
	\begin{enumerate}
		\item 	$X':=\PP(K_E\oplus \cO)$, and $Y':=(2,2)\mbox{ complete intersection in }\PP^4$ (=Theorem 3.4);
		\item The global examples $\{X_1,Y_1\}$ and $\{X_2,Y_2\}$ (=Theorem 4.4, Theorem 5.4).
	\end{enumerate}
\end{thm}
We remark that the above theorem essentially follows from more explicit relationships between their $I-$functions. Let $I^X(q_1,q_2)$ the $I-$functions for $X'$ (resp. $X_1$ or $X_2$) be $I(q_1,q_2)$, for $Y'$ (resp. $Y_1$ or $Y_2$) be $I^Y(y)$. Then one may perform an analytic continuation for $I^X(q_1,q_2)$ to obtain a hypergeometric series $\bar I ^X(x,y)$, while applying the change of variable $x\mapsto q_1^{-1}$ and $y\mapsto q_1q_2$. Then the $I$-function for $Y'$ (resp. $Y_1$ or $Y_2$) is recovered by
\[I^Y(y)=\lim_{x\to 0}\bar I^X(x,y).\]
We will discuss this further in subsequent sections.\\
~\\
\textbf{Plan of the paper.} In Section 2, we begin by introducing the basic definition and background of the quantum $D$-module used in our paper. Then in Section 3, we prove the \Cref{myconj} for the local model. Finally, in Section 4 and 5, we prove \Cref{myconj} for $\{X_1,Y_1\}$ and $\{X_2,Y_2\}$, respectively.\\
~\\
\textbf{Acknowledgement.} I am very grateful to Professors Yongbin Ruan and Yuan-Pin Lee for insightful discussions. I would also like to thank Professors Mark Gross and  Albrecht Klemm for helpful correspondence. Part of this work was done during my visit to the Department of Mathematics at the University of Utah. I would like to thank their kind support and hospitality.

\section{Ambient part quantum $D$-modules}

In this section, we will introduce the language used in our theorem, namely, ambient part quantum $D$-modules. As our target varieties are all complete intersections in toric varieties,  we will also have a short discussion of the roles of $I$-functions played in this picuture. Genereal references on these topics are \cite{MR1677117,MR1492534,MR2003030}.

Let $P$ be a projective toric variety and $Z$ is a complete intersection insider of it. Namely, $Z$ is the zero locus of a section of a vector bundle $\cV$ over $P$. Assume $\cV$ splits, i.e. $\cV=\oplus_{i=1}^r L_i$. 

Let $\cM_{0,n}(P,\beta)$ denote the moduli space of genus-zero stable maps to $P$, where the source curve is a rational nodal curve with $n$-markings and the image is of type $\beta\in H_2(P)$. There exists a virtual fundamental class $[\cM_{0,n}(P,\beta)]^{vir}$. The Gromov-Witten invariants are basically the intersection numbers on $\cM_{0,n}(P,\beta)$. 

In \cite{MR2276766}, the notion of $(\mathbf{e}, \cV)$-twisted Gromov-Witten invariants is introduced. Let $\{T^i\}_{i=0}^N$ be a basis of $H^*(P)$ with $T^0=1$, and $T^1,\cdots, T^r$ form a basis of $H^2(P)$. We consider the following twisted $J$-function
\[J^\cV(t,z^{-1}):=1+\frac{T}{z}+\sum_{n=0}^\infty\frac{Q^d}{n!}(\ev_{n+1})_*\left(\frac{\cV'_{0,n+1,d}}{z(z-\psi_{n+1})}\prod_{i=1}^n\ev^*_i T\right),\]
where $t=\sum t_iT^i$ and $\cV'_{0,n+1,d}$ is the kernel of the map $R^0\pi_*\ev^*_{n+1}\cV\to \ev_{n+1}^*\cV$ and $\pi:\cM_{0,n+1}(P,\beta)\to \cM_{0,n}(P,\beta)$ forgets the last marking. Restricting the above function to $t_{m+1}=\cdots=t_{N}=0$ and set $Q\equiv 1$, $(q):=(q_i)_{i=1}^r=(e^{t_i})_{i=1}^r$, we obtain the twisted small $J$-function $J_{\mathrm{sm}}^{\cV}(t_0,q,z^{-1})$. 

When $c(Z)\geqslant 0$, the mirror theorem in \cite{MR1653024} asserts that
$J^{\cV}(\tau,z^{-1})$ is equal to $I^{\cV}(t,z^{-1})$ up to a mirror transformation $t\mapsto \tau(t)$, where $I^{\cV}(t,z^{-1})$ is an explicit cohomology-valued hypergeometric power series. If $c(Z)\geqslant 0$ fails, then a procedure of Birkhoff factorization is also needed to recover the function $J^{\cV}(t,z^{-1})$ from the twisted $I$-function $I^{\cV}(t,z,z^{-1})$. Therefore, there is a natural identification of the cyclic $D$-module attached to $J^{\cV}(t,z^{-1})$ and $I^{\cV}(t,z,z^{-1})$ (cf. \cite{MR3751817}).

Our definition for the ambient part quantum $D$-module is as follows:
\begin{definition}
	The {\em ambient part quantum $D$-module} of $Z$ is the cyclic $D$-module attached to $\mathbf{e}(\cV)J_\mathrm{sm}^\cV(0,q,z^{-1})$, denoted by $\cH(Z)$ when the ambient variety is understood.
\end{definition}

By Mirror theorem, this can be identified with the cyclic $D$-module generated by the small twisted $I$-function $\mathbf{e}(\cV)I^{\cV}(q,z,z^{-1})$, which arises as solutions to the GKZ system attached to the toric data of $(P,\cV)$ (cf. \cite{MR1677117}). 
\begin{remark}
	By a reconstruction theorem \cite{MR2102400}, when $H^{*}(P)$ is generated by divisors, or when $Z$ is Calabi-Yau, the genus-zero ambient part Gromov-Witten invariants of $Z$ can be recovered from its $J$-function $J^{Z}(t,z^{-1})$. It is related to the $(\mathbf{e},\cV)$-twisted $J$-function $J^{\cV}(t,z^{-1})$ by
\[\mathbf{e}(\cV)J^\cV(t,z^{-1})=\iota_*J^Z(\iota^*t,z^{-1}),\]
where $\iota:Z\hookrightarrow P$ is the inclusion. That's why we consider the cyclic $D$-module attached to $\mathbf{e}(\cV)J^{\cV}(t,z^{-1})$ rather than $J^{\cV}(t,z^{-1})$. A closely related definition of ambient part quantum $D$-modules is introduced by  Mann-Mignon in \cite{MR3663796}. 
\end{remark}

\section{The local model}

In this section, we begin to study the change of quantum $D$-modules associted to the local model of degree-4 transition. Let $E$ be a Del-Pezzo surface in degree 4, which embeds into $\PP^4$ as a $(2,2)$ complete intersection. As described in the intorduction, the local model $\{X',Y'\}$ is the following:

\[X':=\PP(K_E\oplus \cO),\quad Y':=(2,2)\mbox{ complete intersection in }\PP^5.\]

Let $i:E\hookrightarrow \PP^4$ be the embedding. By adjuction formula, we have
\[K_E= i^*(K_{\PP^4}\otimes \cO_{\PP^4}(2)\otimes \cO_{\PP^4}(2))=i^*\cO_{\PP^4}(-1).\]

Thus $X'$ is embedded into $\PP(\cO_{\PP^4}(-1)\oplus \cO_{\PP^4})$ as a complete intersection. Let $\pi: \PP(\cO_{\PP^4}(-1)\oplus \cO_{\PP^4})\to \PP^4$ be the natural projection, then $X'$ can be viewed as the zero locus of a section of $\pi^*(\cO_{\PP^4}(2)\oplus \cO_{\PP^4}(2))$. On the other hand, $Y'$ is the zero locus of a section of $\cO_{\PP^5}(2)\oplus \cO_{\PP^5}(2)$.\\

We will adopt the following notations in this section.
\begin{itemize}
	\item $h:=c_1(\pi^*\cO_{\PP^4}(1)$), corresponding to small parameter $q_1$
	\item Let $\cO_\PP(1)$ be the anti-tautological bundle over $\PP(\cO_{\PP^4}(-1)\oplus \cO_{\PP^4})$, and $\xi:=c_1(\cO_\PP(1))$, corresponding to small parameter $q_2$.
	\item $p:=c_1(\cO_{\PP^5}(1))$, corresponding to small parameter $y$.
\end{itemize}

According to the combinatorical data of $\PP(\cO_{\PP^4}(-1)\oplus \cO_{\PP^4})$ and $\PP^5$, it is straightforward to write down the twisted $I$-functions for $X'$ and $Y'$ as follows.
\[I^{X'}(q_1,q_2): =(2h)(2h)q_1^{h/z}q_2^{\xi/z}\sum_{(d_1,d_2)\in\NN^2} q_1^{d_1}q_2^{d_2}\frac{\prod^0\limits_{m=-\infty} (\xi-h+mz)\prod\limits_{m=1}^{2d_1}(2h+mz)^2}{\prod^{d_1}\limits_{m=1}(h+mz)^5\prod^{d_2}\limits_{m=1}(\xi+mz)\prod^{d_2-d_1}\limits_{m=-\infty}(\xi-h+mz)},\]
subject to the relation $h^5=\xi(\xi-h)=0$,
\[I^{Y'}(y):=(2p)(2p)y^{p/z}\sum_{j\in \NN} \frac{\prod\limits_{m=1}^{2j}(2p+mz)^2}{\prod\limits_{m=1}^{j}(p+mz)^6},\]
subject to the relation $p^6=0$.

To compare these two functions, it is helpful to introduce the following auxcillary hypergeometric series:
\[\bar I^{Y'}(x,y):=(2p)(2p)y^{p/z}\sum_{(i,j)\in \NN^2}x^iy^j\frac{\prod\limits_{m=-\infty}^{0}(p+mz)^5\prod\limits_{m=-\infty}^{2j-2i}(2p+mz)^2}{\prod\limits_{m=-\infty}^{j-i}(p+mz)^5\prod\limits_{m=1}^{j}(p+mz)\prod\limits_{m=1}^{i}(mz)\prod\limits_{m=-\infty}^{0}(2p+mz)^2}.\]
We note that $\bar I^{Y'}(x,y)$ involves two variables $x,y$. It's easy to check that $\bar I^{Y'}$ is a holomorphic function on a small domain minus the origin. We also observe that $\bar I^{Y'}$ has trivial monodromy around $x=0$, thus it makes sense to take the limit $x\to 0$, we obtain
\[\lim_{x\to 0}\bar I^{Y'}(x,y)=I^{Y'}(y),\]
which precisely recovers the twisted $I$-function for $Y'$.

\begin{lemma}
The components of $I^{X'}(q_1,q_2)$ comprise a basis of solutions to the differential equation system $\{\triangle_1I=\triangle_2I=0\}$ at any point around the origin in $(\CC^*)^2$, where
\[\triangle_1:=(z\delta_{q_1})^3-4q_1(2z\delta_{q_1}+z)^2,\]
\[\triangle_2:=z\delta_{q_2}(z\delta_{q_2}-z\delta_{q_1})-q_2.\]
\end{lemma}
\begin{proof}
	First we write 
	$$I^{X'}(q_1,q_2) =(2h)(2h)q_1^{h/z}q_2^{\xi/z}\sum_{(d_1,d_2)\in\NN^2}q_1^{d_1}q_2^{d_2}A_{d_1,d_2}.$$ These two differential operators are obtained precisely by the recursion reltaions between $A_{d_1,d_2}$ and $A_{d_1+1,d_2}$ (or $A_{d_1,d_2+1}$), and the cohomology relation $h^5=\xi(\xi-h)=0$ amounts to the fact that $A_{d_1,d_2}=0$ unless $(d_1,d_2)\in \NN^2$. The components of $I^{X'}(q_1,q_2)$ give rise to 6 linearly-independent solutions to the differential equation system. On the other hand, this differential equation system can have at most 6-dimensional solution space due to a holonomic rank computation. Hence the lemma follows.
\end{proof}
Insipred by the work of Lee-Lin-Wang\cite{MR3751817}, we apply the following change of variable to the above differential equation system
\[q_1\mapsto x^{-1}, \quad q_2\mapsto xy.\]
Then we have the relation
\[\delta_{q_1}=\delta_{y}-\delta_{x},\quad \delta_{q_2}=\delta_{y}.\]
Let $\triangle_1'$, $\triangle_2'$ denote the differential operators obtained by applying the above change of variable, then we have the following lemma
\begin{lemma}
	The components of $\bar I^{Y'}(x,y)$ comprise 4 linearly independent solutions to the differential equation system $\{\triangle'_1I=\triangle'_2I=0\}$ at any point around the origin in $(\CC^*)^2$, where
	\[\triangle'_1:=x(z\delta_{y}-z\delta_{x})^3-4(2(z\delta_y-z\delta_x)+z)^2],\]
	\[\triangle'_2:=(z\delta_{y})(z\delta_{x})-xy.\]
\end{lemma}
\begin{proof}
	This is straightfoward to check.
\end{proof}
To find the extra solutions to the differetial equation system $\{\triangle'_1I=\triangle'_2I=0\}$, we define $\bar I^{Y'}_{ext}(x,y)$ in the following way,
\[\bar I^{Y'}_{ext}(x,y)=x^{\frac12+u}\sum_{(i,j)\in\NN^2}x^iy^jC_{i,j},\tag{3.1}\]
where $\{C_{i,j}\}$ satisfies the following recursion relations for $(i,j)\in \ZZ^2$:
\[C_{i-1,j}(j-i+\frac12-u)^3z=16C_{i,j}(j-i-u)^2,\tag{3.2}\]
\[C_{i-1,j-1}=C_{i,j}(zj)(zi+\frac12+u).\tag{3.3}\] 
\begin{lemma}
	Let $I_5,I_6\in \CC[[x,y,\log x]][x^{\frac12}]$ be the components of $\bar I^{Y'}_{ext}(x,y)$ in the following sense
\begin{align*}
\pi:\CC[[x,y,u,\log x]][x^{\frac12}] &\longrightarrow \CC[[x,y,u,\log x]][x^{\frac12}]/(u^2),\\
x^{\frac12}e^{u\log x}\sum_{(i,j)\in\NN^2}x^iy^jC_{i,j} &\longmapsto I_5+I_6u,
\end{align*}
where $\pi$ is the obvious projection map, and $C_{i,j}$ are defined by the recursion relations (3.2) and (3.3) with initial condition $C_{0,0}=1$.
	Then $I_5$ and $I_6$, together with the components of $\bar I^{Y'}(x,y)$, comprise a basis of solutions to the differential equation system $\{\triangle'_1I=\triangle'_2I=0\}$ at any point around the origin in $(\CC^*)^2$.
\end{lemma}
\begin{proof}
	Given the solution form (3.1), it is straightfoward to check that the recursion relation (3.1) and (3.2) correspond precisely to the differential operator $\triangle'_1$ and $\triangle'_2$, respectively. Choosing initial condition $C_{0,0}=1$, it follows that $C_{i,j}$ are uniquely determined for all $(i,j)\in \NN^2$. If we require $u^2=0$, we see that $C_{i,j}=0$ if $i<0$ or $j<0$. Thus $I_5$ and $I_6$ are solutions to the differential equation system $\{\triangle'_1I=\triangle'_2I=0\}$. It is clear that $I_5$, $I_6$, as well as the components of $\bar I^{Y'}(x,y)$, are all linearly-independent because of their initial terms. On the other hand, the differential equation system $\{\triangle'_1I=\triangle'_2I=0\}$ should have 6-dimensional solution space, hence the lemma is proved.
\end{proof}
Now we are ready to prove the main theorem in this section.
\begin{thm}
	The Conjecture 1.1 holds for the local model $\{X',Y'\}$, namely:
	one may perform analytic continuation of $\cH(X')$ over the extended K\"ahler moduli to obtain a $D$-module $\bar\cH(X')$, then there exists a divisor $E$ and a submodule $\bar\cH^E(X')\subseteq \bar H(X')$ with maximum trivial monodromy around $E$, such that
	\[\bar\cH^E(X')|_{E}\simeq \cH(Y),\]
	where $\bar\cH^E(X')$ is the restriction to $E$.
	\end{thm}
\begin{proof}
	We begin by identifying the ambient part quantum $D$-module $\cH(X')$ and $\cH(Y')$ with the cyclic $D$-modules generated by $I^{X'}(q_1,q_2)$ and $I^{Y'}(y)$, respectively. The change of variable $x\mapsto q_1^{-1}$ and $y\mapsto q_1q_2$ give rise to the anayltic continuation $\cH(X')	\rightsquigarrow \bar\cH(X')$.
	
	By Lemma 3.1 and Lemma 3.2, we may consider the submodule of $\bar\cH(X')$ corresponding to the sub $D$-module generated by the components of $\bar I^{Y'}(x,y)$. It has trivial monodromy around $x=0$ as the initial term of $\bar I^{Y'}(x,y)$ does not involve $x$. This trivial monodromy is also maximal because by Lemma 3.3, the remaining two solutions $I_5, I_6$ have non-trivial monodromy around $x=0$. Let $E$ denote the transition divisor $x=0$, and $\bar\cH^E(X')$ denote this submodule.
	
	Since $I^{Y'}(y)$ is recovered by $I^{Y'}(y)=\lim_{x\to 0} \bar I^{Y'}(x,y)$, we see that $\cH(Y')$ is isomophic to the restriction of $\bar\cH^E(X')$ to $E$. Hence our theorem is proved.
\end{proof}

\section{Global example of type (2,4)}
In this section, we study our first global example of Type II transition in degree 4. Let $Y_1$ be a (2,4) complete intersection in $\PP^5$ defined by the following equations:
\[f(x_1,x_2,x_3,x_4,x_5)+tx_0^2=0,\]
\[x_0^2(g(x_1,x_2,x_3,x_4,x_5)+tx_0^2)+x_0g_1(x_1,x_2,x_3,x_4,x_5)+g_2(x_1,x_2,x_3,x_4,x_5)=0,\]
where $f$ and $g$ are quadratic homogenous polynomial which form a complete intersection, $g_1$ is a generic cubic homogenous polynomial and $g_2$ is a generic quartic homogenous polynomial.

By deforming the above equations to $t=0$, we obtain a singular variety $\bar Y_1\subseteq \PP^5$ defined by 
\[f(x_1,x_2,x_3,x_4,x_5)=0,\]
\[x_0^2g(x_1,x_2,x_3,x_4,x_5)+x_0g_1(x_1,x_2,x_3,x_4,x_5)+g_2(x_1,x_2,x_3,x_4,x_5)=0.\]
Choosing $g_1$ and $g_2$ appropriately, we may assume $\overline{Y_1}$ has a unique singularity at $[1:0:0:0:0:0]$ arised from the (2,2) complete intersection $(f,g)$.

Now we take $X_1$ to be the blow up of $\bar{Y_1}$ at the point $[1:0:0:0:0:0]$. The expectional divisor is given by $\{f=g=0\}$ in $\PP^4$, which is a Del-Pezzo surface in degree 4. Clearly, both $X_1$ and $Y_1$ are Calabi-Yau 3-folds. The transition from $X_1$ to $Y_1$ is our primary example in this section. In our case here, $Y_1$ is the zero locus of a section of $\cO_{\PP^5}(2)\oplus\cO_{\PP^5}(4)$, whereas $X_1$ is natrually embeded into $\PP(\cO_{\PP^4}(-1)\oplus \cO_{\PP^4})$ as a complete intersection.

We adopt the following notations throughout this section.
\begin{itemize}
	\item $h:=c_1(\pi^*\cO_{\PP^4}(1)$), corresponding to small parameter $q_1$
	\item Let $\cO_\PP(1)$ be the anti-tautological bundle over $\PP(\cO_{\PP^4}(-1)\oplus \cO_{\PP^4})$, and $\xi:=c_1(\cO_\PP(1))$, corresponding to small parameter $q_2$.
	\item $p:=c_1(\cO_{\PP^5}(1))$, corresponding to small parameter $y$.
\end{itemize}
Then $X_1$ is the zero locus of a section of the vector bundle $\cO(2h)\oplus\cO(2h+2\xi)$ over $\PP(\cO_{\PP^4}(-1)\oplus \cO_{\PP^4})$, whose $I$-function is the following:
\[I^{X_1}(q_1,q_2): =(2h)(2h+2\xi)q_1^{h/z}q_2^{\xi/z}\sum_{(d_1,d_2)\in\NN^2} q_1^{d_1}q_2^{d_2}\frac{\prod^0\limits_{m=-\infty} (\xi-h+mz)\prod\limits_{m=1}^{2d_1}(2h+mz)\prod\limits_{m=1}^{2d_1+2d_2}(2h+2\xi+mz)}{\prod^{d_1}\limits_{m=1}(h+mz)^5\prod^{d_2}\limits_{m=1}(\xi+mz)\prod^{d_2-d_1}\limits_{m=-\infty}(\xi-h+mz)},\]
subject to the relation $h^5=\xi(\xi-h)=0$.

On the other hand, the $I$-function for $Y_1$ is the following:
\[I^{Y_1}(y):=(2p)(4p)y^{p/z}\sum_{j\in \NN} \frac{\prod\limits_{m=1}^{4j}(4p+mz)\prod\limits_{m=1}^{2j}(2p+mz)}{\prod\limits_{m=1}^{j}(p+mz)^6},\]
subject to the relation $p^6=0$.

Similar to the local model, we introduce the following auxcillary hypergeometric series $\bar I^{Y_1}$ in two variables $x$ and $y$.
\[\bar I^{Y_1}(x,y):=(2p)(4p)y^{p/z}\sum_{(i,j)\in \NN^2}x^iy^j\frac{\prod\limits_{m=-\infty}^{0}(p+mz)^5\prod\limits_{m=-\infty}^{2j-2i}(2p+mz)\prod\limits_{m=-\infty}^{4j-2i}(2p+mz)}{\prod\limits_{m=-\infty}^{j-i}(p+mz)^5\prod\limits_{m=1}^{j}(p+mz)\prod\limits_{m=1}^{i}(mz)\prod\limits_{m=-\infty}^{0}(2p+mz)^2}.\]
It's easy to check that $\bar I^{Y_1}$ is a holomorphic function on a small domain minus the origin, and it has trivial monodromy around $x=0$. Taking the limit $x\to 0$, we obtain
\[\lim_{x\to 0}\bar I^{Y_1}(x,y)=I^{Y_1}(y),\]
which precisely recovers the $I$-function for $Y_1$.

To study the relation between the ambient part quantum $D$-modules of $X_1$ and $Y_1$, we want to find a way to relate $I^{X_1}(x,y)$ and $I^{Y_1}(y)$. As $I^{Y_1}$ is recovered from $\bar I^{Y_1}$, it is tempting to study the relation between $I^{X_1}$ and $I^{Y_1}$ as they both involve 2 parameters. 

We consider the Picard-Fuchs equations that annihilate $I^{X'}$, which usually originates from a GKZ system attached to the toric data. We have the following lemma:

\begin{lemma}
		The components of $I^{X_1}(q_1,q_2)$ comprise a basis of solutions to the differential equation system $\{\triangle_1I=\triangle_2I=\cL I=0\}$ at any point around the origin in $(\CC^*)^2$, where
		\[\triangle_1:=(\delta_{q_1})^5-4(\delta_{q_1})(\delta_{q_1}+\delta_{q_2})(2\delta_{q_1}-1)(\delta_{q_2}-\delta_{q_1}+1)(2\delta_{q_1}+2\delta_{q_2}-1)q_1,\]
		\[\triangle_2:=\delta_{q_2}(\delta_{q_2}-\delta_{q_1})-2(\delta_{q_1}+\delta_{q_2})(2\delta_{q_1}+2\delta_{q_2}-1)q_2,\]
		\[\cL:=(2\delta_{q_1}^3-2\delta_{q_1}^2\delta_{q_2}+\delta_{q_1}\delta_{q_2}^2)-8(2\delta_{q_1}-1)(\delta_{q_2}-\delta_{q_1}+1)(2\delta_{q_2}+2\delta_{q_1}-1)q_1-2\delta_{q_1}\delta_{q_2}(2\delta_{q_1}+2\delta_{q_2}-1)q_2.\]
		Moreover, these differential operators are related in the following factorization
		\[2\triangle_1+\delta_{q_1}^2\delta_{q_2}\triangle_2=(\delta_{q_1}+\delta_{q_2})\delta_{q_1}\cL.\tag{4.1}\]
\end{lemma}
\begin{proof}
	Indeed, the GKZ system attached to the toric data yields the generators $\triangle_1$ and $\triangle_2$. By the factorization (4.1), we obtain a differential operator $\cL$ of order 3, thus the system $\{\triangle_1I=\triangle_2I=\cL I=0\}$ can have at most 6-dimensional solution space. It is direct to check that the components of $I^{X_1}$ give 6 linearly independent solution to this differential equation system, hence they must form a basis.
\end{proof}

Similar to Section 3, we apply the following change of variables \[q_1\mapsto x^{-1}, \quad q_2\mapsto xy,\]
which yields the following relations between differential operators
\[\delta_{q_1}=\delta_{y}-\delta_{x},\quad \delta_{q_2}=\delta_{y}.\]
Let $\triangle_1'$, $\triangle_2'$, $\cL'$ denote the differential operators obtained by applying the above change of variables to $\triangle_1,\triangle_2,\cL$, respectively. Then we have
\begin{lemma}
	The components of $\bar I^{Y_1}(x,y)$ comprise 4 linearly independent solutions to the differential equation system $\{\triangle'_1I=\triangle'_2I=\cL' I=0\}$ at any point around the origin in $(\CC^*)^2$, where
	\[\triangle'_1:=(\delta_{y}-\delta_{x})^5-4(\delta_{y}-\delta_{x})(2\delta_{y}-\delta_{x})(2\delta_{y}-2\delta_{x}-1)(\delta_{x}+1)(4\delta_{y}-2\delta_{x}-1)x^{-1},\]
	\[\triangle'_2:=\delta_{y}\delta_{x}-2(2\delta_{y}-\delta_{x})(4\delta_{y}-2\delta_{x}-1)xy,\]
	\begin{align*}
	\cL'&:=(2(\delta_{y}-\delta_x)^3-2(\delta_{y}-\delta_x)^2\delta_{y}+(\delta_{y}-\delta_x)\delta_{y}^2)-8(2\delta_{y}-2\delta_x-1)(\delta_{x}+1)(4\delta_{y}-2\delta_{x}-1)q_1\\&\qquad-2(\delta_{y}-\delta_x)\delta_{y}(4\delta_{y}-2\delta_{x}-1)xy.
	\end{align*}
\end{lemma}
\begin{proof}
	This can be checked directly.
\end{proof}
To find the extra solutions to the above differetial equation system $\{\triangle'_1I=\triangle'_2I=0\}$, we introduce $\bar I^{Y_1}_{ext}(x,y)$ in the following way,
\[\bar I^{Y_1}_{ext}(x,y)=x^{\frac12+u}\sum_{(i,j)\in\NN^2}x^iy^jC_{i,j},\tag{4.2}\]
where $\{C_{i,j}\}$ satisfies the following recursion relations for $(i,j)\in \ZZ^2$:
\[C_{i,j}(j-i-\frac12-u)^4=C_{i+1,j}(2j-u-i-\frac12)(2j-2u-2i-2)(u+i+\frac32)(4j-2u-2i-2),\tag{4.3}\]
\[(4j-2u-2i-1)(4j-2i-2u-2)C_{i-1,j-1}=C_{i,j}(j)(i+\frac12+u).\tag{4.4}\] 
\begin{lemma}
	Let $I_5,I_6\in \CC[[x,y,\log x]][x^{\frac12}]$ be the components of $\bar I^{Y_1}_{ext}(x,y)$ in the following sense
	\begin{align*}
	\pi:\CC[[x,y,u,\log x]][x^{\frac12}] &\longrightarrow \CC[[x,y,u,\log x]][x^{\frac12}]/(u^2),\\
	x^{\frac12}e^{u\log x}\sum_{(i,j)\in\NN^2}x^iy^jC_{i,j} &\longmapsto I_5+I_6u,
	\end{align*}
	where $\pi$ is the obvious projection map, and $C_{i,j}$ are defined recursively by (4.3) and (4.4) with initial condition $C_{0,0}=1$.
	Then $I_5$ and $I_6$, together with the components of $\bar I^{Y_1}(x,y)$, comprise a basis of solutions to the differential equation system $\{\triangle'_1I=\triangle'_2I=\cL'I=0\}$ at any point around the origin in $(\CC^*)^2$.
\end{lemma}
\begin{proof}
	Given the solution form (4.2), it is straightfoward to check that the recursion relation (4.3) and (4.4) are compatible with $\triangle'_1$, $\triangle'_2$ and $\cL'$. If we require $u^2=0$, we see that $C_{i,j}=0$ if $i<0$ or $j<0$. Given initial condition $C_{0,0}=1$, it is clear that $C_{i,j}$ are uniquely determined for all $(i,j)\in \NN^2$. Thus $I_5$ and $I_6$ are solutions to the differential equation system $\{\triangle'_1I=\triangle'_2I=\cL'I=0\}$. We also note that $I_5$, $I_6$, together with the components of $\bar I^{Y_1}(x,y)$, are linearly-independent because of their initial terms. On the other hand, the differential equation system $\{\triangle'_1I=\triangle'_2I=\cL'I=0\}$ should have at most 6-dimensional solution space, hence the lemma follows.
\end{proof}
We are now in a position to prove the main theorem in this section.
\begin{thm}
	The Conjecture 1.1 holds for the $\{X_1,Y_1\}$, namely:
	one may perform analytic continuation of $\cH(X_1)$ over the extended K\"ahler moduli to obtain a $D$-module $\bar\cH(X_1)$, then there exists a divisor $E$ and a submodule $\bar\cH^E(X_1)\subseteq \bar H(X_1)$ with maximum trivial monodromy around $E$, such that
	\[\bar\cH^E(X_1)|_{E}\simeq \cH(Y_1),\]
	where $\bar\cH^E(X_1)$ is the restriction to $E$.
\end{thm}
\begin{proof}
	Following the argument in Section 3, we first identify the ambient part quantum $D$-module $\cH(X_1)$ and $\cH(Y_1)$ with the cyclic $D$-modules generated by $I^{X_1}(q_1,q_2)$ and $I^{Y_1}(y)$, respectively. The change of variable $x\mapsto q_1^{-1}$ and $y\mapsto q_1q_2$ give rise to the anayltic continuation $\cH(X_1)	\rightsquigarrow \bar\cH(X_1)$.
	
	By Lemma 4.1 and Lemma 4.2, we consider the submodule of $\bar\cH(X_1)$ corresponding to the sub $D$-module attached to the components of $\bar I^{Y_1}(x,y)$. It has trivial monodromy around $x=0$ as the initial term of $\bar I^{Y_1}(x,y)$ does not involve $x$. This trivial monodromy is also maximal because by Lemma 4.3, it is clear that the remaining two solutions $I_5, I_6$ have non-trivial monodromy around $x=0$. Let $E$ denote the transition divisor corresponding to $x=0$, and $\bar\cH^E(X_1)$ denote this submodule.
	
	As $I^{Y_1}(y)$ is recovered by $I^{Y_1}(y)=\lim_{x\to 0} \bar I^{Y_1}(x,y)$, we obtain immediately that $\cH(Y_1)$ is isomophic to the restriction of $\bar\cH^E(X_1)$ to $E$. Hence the theorem is proved.
\end{proof}

\section{Global example of type (3,3)}
In this section, our goal is to verify Conjecuture 1.1 for another global example of degree-4 transition. Let $Y_2$ be a $(3,3)$-complete intersection in $\PP^5$ defined by the following two equations:
\[x_0(f(x_1,x_2,x_3,x_4,x_5)+tx_0^2)+f_1(x_1,x_2,x_3,x_4,x_5)=0,\]
\[x_0(g(x_1,x_2,x_3,x_4,x_5)+tx_0^2)+g_1(x_1,x_2,x_3,x_4,x_5)=0,\]
where $f$ and $g$ are quadratic homogenous polynomials that form a complete interesection, while $f_1$ and $g_1$ are generic homogenous polynomial in degree 3.
Deforming the above equations by letting $t=0$, we obtain a singular 3-fold $\bar Y_2\subseteq \PP^5$, whose defining equations are:
\[x_0f(x_1,x_2,x_3,x_4,x_5)+f_1(x_1,x_2,x_3,x_4,x_5)=0,\]
\[x_0g(x_1,x_2,x_3,x_4,x_5)+g_1(x_1,x_2,x_3,x_4,x_5)=0.\]
As $f_1$ and $g_1$ are choosen to be generic, we may assume that $\bar Y_2$ has a unique singularity at $[1:0:0:0:0:0]$, which is a (2,2) complete intersection singularity. If we blow up $Y_0$ at this point, we obtain a smooth Calabi-Yau 3-fold $X_2$, where the exceptional divisor is a (2,2)-complete intersection in $\PP^4$, namely, a Del-Pezzo surface in degree 4. 

Going backwards, we see that $Y_2$ is obtained by birationally contracting a degree-4 Del-Pezzo surface in $X_2$ and followed by smoothing out the singularity. Thus $\{X_2,Y_2\}$ is another global example of degree-4 transitions, where both sides are Calabi-Yau 3-folds. In this case, $Y_2$ is the zero locus of a section of $\cO_{\PP^5}(3)\oplus\cO_{\PP^5}(3)$, whereas $X_2$ natrually sits inside $\PP(\cO_{\PP^4}(-1)\oplus \cO_{\PP^4})$ as a complete intersection.

As before, we shall adopt the following notations throughout this section.
\begin{itemize}
	\item $h:=c_1(\pi^*\cO_{\PP^4}(1)$), corresponding to small parameter $q_1$
	\item Let $\cO_\PP(1)$ be the anti-tautological bundle over $\PP(\cO_{\PP^4}(-1)\oplus \cO_{\PP^4})$, and $\xi:=c_1(\cO_\PP(1))$, corresponding to small parameter $q_2$.
	\item $p:=c_1(\cO_{\PP^5}(1))$, corresponding to small parameter $y$.
\end{itemize}
Then $X_2$ is the zero locus of a section of the vector bundle $\cO(2h+\xi)\oplus\cO(2h+\xi)$ over $\PP(\cO_{\PP^4}(-1)\oplus \cO_{\PP^4})$, whose $I$-function is the following:
\[I^{X_2}(q_1,q_2): =(2h+\xi)^2q_1^{h/z}q_2^{\xi/z}\sum_{(d_1,d_2)\in\NN^2} q_1^{d_1}q_2^{d_2}\frac{\prod^0\limits_{m=-\infty} (\xi-h+mz)\prod\limits_{m=1}^{2d_1}(2h+mz)\prod\limits_{m=1}^{2d_1+d_2}(2h+\xi+mz)^2}{\prod^{d_1}\limits_{m=1}(h+mz)^5\prod^{d_2}\limits_{m=1}(\xi+mz)\prod^{d_2-d_1}\limits_{m=-\infty}(\xi-h+mz)},\]
subject to the relation $h^5=\xi(\xi-h)=0$.

On the other hand, the $I$-function for $Y_2$ is the following:
\[I^{Y_2}(y):=(3p)^2y^{p/z}\sum_{j\in \NN} \frac{\prod\limits_{m=1}^{3j}(3p+mz)^2}{\prod\limits_{m=1}^{j}(p+mz)^6},\]
subject to the relation $p^6=0$.

To compare $I^{X_2}$ and $I^{Y_2}$, we introduce the following auxcillary hypergeometric series $\bar I^{Y_2}$ in two variables $x$ and $y$.
\[\bar I^{Y_2}(x,y):=(3p)^2y^{p/z}\sum_{(i,j)\in \NN^2}x^iy^j\frac{\prod\limits_{m=-\infty}^{0}(p+mz)^5\prod\limits_{m=-\infty}^{3j-2i}(3p+mz)^2}{\prod\limits_{m=-\infty}^{j-i}(p+mz)^5\prod\limits_{m=1}^{j}(p+mz)\prod\limits_{m=1}^{i}(mz)\prod\limits_{m=-\infty}^{0}(2p+mz)^2}.\]
It is clear that $\bar I^{Y_2}$ is a holomorphic function on a small domain minus the origin, and it has trivial monodromy around $x=0$. Taking the limit $x\to 0$, we obtain
\[\lim_{x\to 0}\bar I^{Y_2}(x,y)=I^{Y_2}(y),\]
which still recovers the $I$-function for $Y_2$.

To compare $I^{X_2}$ and $I^{Y_2}$, we turn our attention to the relation betwen $I^{X_2}$ and $\bar I^{Y_2}$, as $\bar I^{Y_2}$ involves two variables $x$ and $y$, and recovers $I^{Y_2}$ naturally. We consider the Picard-Fuchs equations that annihilates $I^{X_2}$, which usually arises from a GKZ system attached to the toric data.
\begin{lemma}
	The components of $I^{X_2}(q_1,q_2)$ comprise a basis of solutions to the differential equation system $\{\triangle_1I=\triangle_2I=\cL I=0\}$ at any point around the origin in $(\CC^*)^2$, where
	\[\triangle_1:=(\delta_{q_1})^5-q_1(\delta_{q_2}-\delta_{q_1})(2\delta_{q_1}+\delta_{q_2}+1)^2(2\delta_{q_1}+\delta_{q_2}+2)^2,\]
	\[\triangle_2:=\delta_{q_2}(\delta_{q_2}-\delta_{q_1})-q_2(2\delta_{q_2}+\delta_{q_1}+1)^2,\]
	\[\cL:=9\delta_{q_1}^3-5\delta_{q_2}^3-36(\delta_{q_2}-\delta_{q_1}+1)(2\delta_{q_1}+\delta_{q_2}-1)^2q_1+(36\delta_{q_1}^3+45\delta_{q_1}^2\delta_{q_2}+25\delta_{q_1}\delta_{q_2}^2+5\delta_{q_2}^3)q_2.\]
	Moreover, these differential operators are related in the following factorization
	\[36\triangle_1-(36\delta_{q_1}^3+45\delta_{q_1}^2\delta_{q_2}+25\delta_{q_1}\delta_{q_2}^2+5\delta_{q_2}^3)\triangle_2=(2\delta_{q_1}+\delta_{q_2})^2\cL.\tag{5.1}\]
\end{lemma}
\begin{proof}
	The GKZ system attached to the toric data gives rise to the generators $\triangle_1$ and $\triangle_2$. By the factorization (5.1), we obtain a differential operator $\cL$ of order 3, thus the system $\{\triangle_1I=\triangle_2I=\cL I=0\}$ can have at most 6-dimensional solution space. It is straightforward to check that the components of $I^{X_2}$ comprise 6 linearly independent solution to this differential equation system, hence they must form a basis.
\end{proof}
On the other hand, by making the change of variable $x\mapsto q_1^{-1}$ and $y\mapsto q_1q_2$, the resulting differential operators turn out to be the annihilators of $\bar I^{Y_2}$. We have the following lemma.
\begin{lemma}
	The components of $\bar I^{Y_2}(x,y)$ comprise 4 linearly independent solutions to the differential equation system $\{\triangle'_1I=\triangle'_2I=\cL' I=0\}$ at any point around the origin in $(\CC^*)^2$, where
	\[\triangle'_1:=x(\delta_{y}-\delta_{x})^5-\delta_{x}(3\delta_{y}-2\delta_{x}+1)^2(3\delta_{y}-2\delta_{x}+2)^2,\]
	\[\triangle'_2:=\delta_{y}\delta_{x}-xy(3\delta_{y}-2\delta_{x}+1)^2,\]
	\begin{align*}
	\cL'&:=9(\delta_{y}-\delta_x)^3-5\delta_y^3-36(\delta_x+1)(3\delta_{y}-2\delta_x-1)^2x^{-1}+(36(\delta_{y}-\delta_x)^3+45(\delta_{y}-\delta_x)^2\delta_y\\&\qquad+25(\delta_{y}-\delta_x)\delta_{y}^2+5\delta_{y}^3)xy.
	\end{align*}
\end{lemma}
\begin{proof}
	This is easy to check.
\end{proof}

To find the extra solutions to the above differential equation system, we adopt the same method used in previous sections, namely, define $\bar I^{Y_2}_{ext}(x,y)$ in the following way.
\[\bar I^{Y_2}_{ext}(x,y)=x^{\frac12+u}\sum_{(i,j)\in\NN^2}x^iy^jC_{i,j},\tag{5.2}\]
where $\{C_{i,j}\}$ satisfies the following recursion relations for $(i,j)\in \ZZ^2$:
\[C_{i-1,j}(j-u+\frac12)^5=C_{i,j}(i+u+\frac12)(3j-2i-2u)(3j-2i-2u+1)^2,\tag{5.3}\]
\[(3j-2i-2u-1)^2C_{i-1,j-1}=C_{i,j}(j)(i+\frac12+u).\tag{5.4}\] 
\begin{lemma}
	Let $I_5,I_6\in \CC[[x,y,\log x]][x^{\frac12}]$ be the components of $\bar I^{Y_2}_{ext}(x,y)$ in the following sense
	\begin{align*}
	\pi:\CC[[x,y,u,\log x]][x^{\frac12}] &\longrightarrow \CC[[x,y,u,\log x]][x^{\frac12}]/(u^2),\\
	x^{\frac12}e^{u\log x}\sum_{(i,j)\in\NN^2}x^iy^jC_{i,j} &\longmapsto I_5+I_6u,
	\end{align*}
	where $\pi$ is the obvious projection map, and $C_{i,j}$ are defined recursively by (5.3) and (5.4) with initial condition $C_{0,0}=1$.
	Then $I_5$ and $I_6$, together with the components of $\bar I^{Y_2}(x,y)$, comprise a basis of solutions to the differential equation system $\{\triangle'_1I=\triangle'_2I=\cL'I=0\}$ at any point around the origin in $(\CC^*)^2$.
\end{lemma}
\begin{proof}
	As $\bar I^{Y_1}_{ext}(x,y)$ is of form (5.2), it is easy to check that the recursion relation (5.3) and (5.4) are compatible with $\triangle'_1$, $\triangle'_2$ and $\cL'$. If we require $u^2=0$, we see that $C_{i,j}=0$ if $i<0$ or $j<0$. Moreover, the initial condition $C_{0,0}=1$ allows us to determine $C_{i,j}$ uniquely for all $(i,j)\in \NN^2$. Thus $I_5$ and $I_6$ are solutions to the differential equation system $\{\triangle'_1I=\triangle'_2I=\cL'I=0\}$. We also note that $I_5$, $I_6$, together with the components of $\bar I^{Y_2}(x,y)$, are linearly-independent because of their initial terms. On the other hand, the differential equation system $\{\triangle'_1I=\triangle'_2I=\cL'I=0\}$ should have at most 6-dimensional solution space, hence the lemma follows.
\end{proof}
Now we arrive at the proof of the main theorem in this section.
\begin{thm}
	The Conjecture 1.1 holds for the $\{X_2,Y_2\}$, namely:
	one may perform analytic continuation of $\cH(X_2)$ over the extended K\"ahler moduli to obtain a $D$-module $\bar\cH(X_2)$, then there exists a divisor $E$ and a submodule $\bar\cH^E(X_2)\subseteq \bar H(X_2)$ with maximum trivial monodromy around $E$, such that
	\[\bar\cH^E(X_2)|_{E}\simeq \cH(Y_2),\]
	where $\bar\cH^E(X_2)$ is the restriction to $E$.
\end{thm}
\begin{proof}
	Following the same line of arguments in previous sections, we identify the ambient part quantum $D$-module $\cH(X_2)$ and $\cH(Y_2)$ with the cyclic $D$-modules generated by $I^{X_2}(q_1,q_2)$ and $I^{Y_2}(y)$, respectively. The change of variable $x\mapsto q_1^{-1}$ and $y\mapsto q_1q_2$ give rise to the anayltic continuation $\cH(X_2)	\rightsquigarrow \bar\cH(X_2)$.
	
	By Lemma 5.1 and Lemma 5.2, we may consider the submodule of $\bar\cH(X_2)$ corresponding to the sub $D$-module attached to the components of $\bar I^{Y_2}(x,y)$. It has trivial monodromy around $x=0$ as the initial term of $\bar I^{Y_2}(x,y)$ does not involve $x$. This trivial monodromy is also maximal because by Lemma 5.3, the remaining two solutions $I_5, I_6$ have non-trivial monodromy around $x=0$. Let $E$ denote the transition divisor $x=0$, and $\bar\cH^E(X_1)$ denote this submodule.
	
	Since $I^{Y_2}(y)$ is recovered by $I^{Y_2}(y)=\lim_{x\to 0} \bar I^{Y_2}(x,y)$, we conclude that $\cH(Y_2)$ is recovered as the restriction of $\bar\cH^E(X_2)$ to $E$. Hence the theorem is proved.
\end{proof}

\bibliography{research}

\bibliographystyle{plain}

\end{document}